\documentclass[12pt]{article}
\usepackage{verbatim}
\usepackage{amssymb}
\usepackage{amsthm}

\usepackage{amsmath}
\usepackage{mathrsfs}
\usepackage[margin=1in]{geometry}
\usepackage[usenames,dvipsnames] {color}

\definecolor{dark_green}{rgb}{0.0, 0.7, 0.0}

\numberwithin{equation}{section}

\title{A generalization of Hardy's operator and an asymptotic M\"untz-Sz\'asz Theorem}

\newcommand{\mc}{M\raise.45ex\hbox{c}Carthy}
\author{Jim Agler
\and
John E. M\raise.5ex\hbox{c}Carthy
\thanks{Partially supported by National Science Foundation Grant  
DMS 2054199}
}

\def\h{\mathcal{H}}

\def\be{\begin{equation}}
\def\ee{\end{equation}}
\def\norm#1{\| #1 \|}
\def\m{\mathcal{M}}

\newcommand\cmi{{\mathcal M}_\infty^{\rm unif}}

\def\h{\mathcal{H}}

\def\ltwo{\ell^2}
\def\set#1#2{\{ #1 \, | \, #2\}}

\def\ltwo{{\rm L}^2([0,1])}
\def\ltwos{{\rm L}^2([s,1])}

\def\ip#1#2{\langle\, #1\, ,#2\, \rangle}

\newcommand{\twopartdef}[4]
{
	\left\{
		\begin{array}{ll}
			#1 & \mbox{if } #2 \\
			#3 & \mbox{if } #4
		\end{array}
	\right.
}

\DeclareMathOperator{\spn}{span}
\DeclareMathOperator{\dist}{dist}

\renewcommand\={\ = \ }
\newcommand\M{{\mathcal M}}

\begin{document}

\bibliographystyle{plain}
\theoremstyle{definition}
\newtheorem{defin}[equation]{Definition}
\newtheorem{lem}[equation]{Lemma}
\newtheorem{prop}[equation]{Proposition}
\newtheorem{thm}[equation]{Theorem}
\newtheorem{claim}[equation]{Claim}
\newtheorem{ques}[equation]{Question}
\newtheorem{fact}[equation]{Fact}
\newtheorem{axiom}[equation]{Technical Axiom}
\newtheorem{newaxiom}[equation]{New Technical Axiom}
\newtheorem{cor}[equation]{Corollary}
\newtheorem{exam}[equation]{Example}
\newtheorem{rem}[equation]{Remark}
\maketitle

\section{Overview} In this note we shall give a novel proof that Hardy's Operator $H$, defined on $\ltwo$ by the formula,
\[
Hf (x) \ =\ \frac{1}{x} \int_0^x f(t) dt,\qquad x\in (0,1],
\]
is bounded. This new proof relies only on algebra together with the observation that the monomial functions are eigenvectors for $H$. Specifically, for each $k\ge 0$,
\be\label{over.10}
H x^k = \frac{1}{k+1}x^k.
\ee
Always intrigued by results in analysis whose proofs rely mainly on algebra, the new proof of Hardy's Inequality prompts the authors to propose the following definition.
\begin{defin}\label{def.over.10}
We say that $T$ is a \emph{monomial operator} if $T$ is a bounded operator on $\ltwo$ and there exist a  number $m$ and a sequence of scalars $c_0,c_1,c_2\ldots$ such that for all $k$,
\be\label{over.20}
T x^k =c_k x^{k+m}.
\ee
\end{defin}
We shall call the number $m$ in \eqref{over.20} the {\em order} of $T$. It can be any complex number with non-negative real part, though in all our examples it will be a natural number.
In addition to $H$, a monomial operator of order $0$, 
other examples of monomial operators are the multiplication operator $M_x$ that sends a function $f(x)$ 
to the function $x f(x)$, and the Volterra operator $V = M_x H$, the  operator given by
\[
Vf (x) \ =\ \int_0^x f(t) dt,\qquad x\in (0,1].
\]
Both $M_x$ and $V$ are of order $1$.

For $ 0 \leq s \leq 1$ we shall use $L^2[s,1]$ to denote the closed subspace of $L^2[0,1]$ of functions that
are $0$ on $[0,s]$.
\begin{defin}
We shall say that a bounded operator $T: L^2[0,1] \to L^2[0,1]$ is {\em vanishing preserving}
if $T L^2[s,1] \subseteq L^2[s,1]$ for every $s$ in $[0,1)$.
\end{defin}

In this note we shall prove the following result.
\begin{thm}\label{thm.over.10}
If $T$ is a monomial operator, then $T$ is vanishing preserving.
\end{thm}
Why might such a theorem be true? If $T$ is a monomial operator, and $f$ is a polynomial  that
vanishes at $0$ to some high order $M$, then $Tf$ also vanishes to order at least $M$. So if one thinks of 
vanishing on $[0,s]$ as an extreme case of vanishing to high order, one might believe that monomial operators preserve this property.



Our proof of Theorem \ref{thm.over.10}  relies on a new type of M\"untz-Sz\'asz Theorem, wherein the monomial sequence is allowed to drift. This may be more interesting than the theorem itself!

\section{Hardy's Inequality}
For a continuous function $f$ on $[0,1]$ consider the continuous function $Hf$ on $(0,1]$ defined by the formula
\be\label{hard10}
Hf (x) =\frac{1}{x} \int_0^x f(t) dt,\qquad x\in (0,1].
\ee
Noting that as $x \to 0$, $\frac{1}{x} \to \infty$ and $\int_0^x f(t)dt \to 0$, the following question arises:
\[
\text{How does $Hf$ behave near 0?}
\]

Invoking the Mean Value Theorem for Integrals yields that for each $x\in [0,1]$, there exists $c\in [0,x]$ such that $Hf(x)=f(c)$. Thus,
\be\label{hard20}
|Hf(x)|\le \max_{t\in [0,x]}|f(t)|,\qquad x\in (0,1],
\ee
so that in particular, $Hf$ is bounded near 0.
More delicately, if we apply the MVT to the function $f(x)-f(0)$, we obtain the estimate
\be\label{hard30}
|Hf(x)-f(0)|\le \max_{t \in [0,x]}|f(t)-f(0)|,\qquad x\in (0,1],
\ee
which implies that $Hf(x) \to f(0)$ as $x \to 0$. Therefore, if we agree to extend the definition of $Hf$ at the point $x=0$ by setting $Hf(0)=f(0)$, then our observations imply the following proposition.
\begin{prop}
If $f$ is a continuous function on $[0,1]$, then $Hf$ is a continuous function on $[0,1]$. Furthermore,
\be
\max_{x \in [0,1]} |Hf(x)|\le \max_{x \in [0,1]} |f(x)|.
\ee
\end{prop}

Hardy \cite{har20} was the first to study the local behavior of $H$ at $0$ for functions  equipped with norms other than the supremum norm. His result when specialized to $\ltwo$, the Hilbert space of square integrable Lesbesgue measurable functions on $[0,1]$, is as follows.
\begin{prop}{\bf (Hardy's Inequality in $\ltwo$)}
\label{hard.prop.10}
If $f$ is a measurable function on $[0,1]$, then $Hf$ is a measurable function on $[0,1]$, and
\be
\label{eqH4}
\int_0^1 |Hf(x)|^2 dx \le 4 \int_0^1 |f(x)|^2 .
\ee
\end{prop}

A linear operator $T$ on a normed vector space ${\mathcal V}$ is called {\em bounded}\footnote{It is a straightforward exercise to show that a linear operator is bounded if and only if it is continuous.} if there is some constant $C$ so that
\be
\label{bded}
\| T v \| \ \leq \ C \| v \| \qquad \forall v \in {\mathcal V} .
\ee
The infimum of all $C$ for which \eqref{bded} holds is called the norm of $T$, and written $\| T \|$.
Using this terminology, 
 \eqref{eqH4} says $\norm{H}\le 2$.

Our proof of Proposition \ref{hard.prop.10} in Section \ref{secd} relies on  a new ``sum of squares identity" involving the operator $H$, proved using \eqref{over.10}.
\section{ Hilbert Space distance formula}
Let $h_1,\ldots,h_n$ be $n$ vectors in a hilbert space $\h$. We may associate to these vectors their \emph{Gram matrix}, i.e.,  the $n \times n$ matrix $G[h_1,\ldots,h_n]$ defined by
\[
\big[ G[h_1,\ldots,h_n] \big]_{ij} \= \big[\ \ip{h_j}{h_i}_{\h}\ \big].
\]
An often used application is the following elegant formula for the distance to the span of $h_1,\ldots,h_n$.
\begin{thm}\label{gram.thm.10}
{\bf (Hilbert Space Distance Formula)}
If $\h$ is a Hilbert space, $h\in \h$, and $h_1,h_2,\ldots, h_N \in\h$ are linearly independent, then
\be
\label{eqc10}
\dist(h,\spn \{h_1,h_2,\ldots,h_N\})\=\sqrt{\frac{\det G[h,h_1,h_2,\ldots,h_N]}{\det G[h_1,h_2,\ldots,h_N]}}.
\ee
\end{thm}
\begin{proof}
Write
$h = k+ m$, where $k$ is in the span of $\{ h_1, \dots, h_N \}$ and $m$ is perpendicular to the span.
Then $\| m\| = \dist(h,\spn \{h_1,h_2,\ldots,h_N\})$.

We can write
\[
\det G[h,h_1,h_2,\ldots,h_N] \= \det G[k,h_1,h_2,\ldots,h_N] + \det G[m ,h_1,h_2,\ldots,h_N] .
\]
Since $k$ is in the span of $\{ h_1, \dots, h_N \}$, we have
\[
\det G[k,h_1,h_2,\ldots,h_N] \ = \ 0.
\]
Moreover,
$ G[m ,h_1,h_2,\ldots,h_N]$ is a  matrix whose first row is
$( \ip{m}{m}, 0 , \dots, 0 )$. Therefore
\[
\det G[m ,h_1,h_2,\ldots,h_N] \= \| m \|^2 \det G[h_1,h_2,\ldots,h_N .
\]
Combining these observations, we get \eqref{eqc10}.
\end{proof}
\section{A Hilbert space proof of Hardy's Inequality}
\label{secd}

The key step in our proof is the following lemma.
\begin{lem}{\bf An Identity for $H$}
\label{lemd1}
\[
\norm{f}^2 \= \norm{(1-H)f}^2 + \left|\int_0^1 f(x) dx\right|^2 \qquad \forall f \in \ltwo.
\]
\end{lem}
\begin{proof}
It suffices to show for $f$ a polynomial, since the polynomials are dense in $L^2[0,1]$. If
\[
f(x)\ =\sum_{j=0}^n a_j x^j
\]
then 
\begin{eqnarray*}
\norm{f}^2  &\= & \int_0^1 \sum_{i,j=0}^n a_i \overline{a_j} x^{i+j} dx \\
&=& \sum_{i,j=0}^n a_i \overline{a_j} \frac{1}{i+j+1} .
\end{eqnarray*}
Likewise, as
\[
(1-H) f(x) \= \sum_{j=0}^n \frac{j}{j+1} a_j x^j,
\]
\[
\norm{(1-H)f}^2 \=     \sum_{i,j=0}^n a_i \overline{a_j}\frac{ij}{(i+1)(j+1)(i+j+1)}.
\]
Hence
\begin{align*}
\norm{f}^2-\norm{(1-H)f}^2 &\=
\sum_{i,j=0}^n a_i \overline{a_j} \left( \frac{1}{i+j+1} - \frac{ij}{(i+1)(j+1)(i+j+1)}\right)\\
&\= 
\sum_{i,j=0}^n a_i \overline{a_j} \left( \frac{1}{(i+1)(j+1)} \right)\\
&\=|\int_0^1 f(x)dx|^2.
\end{align*}
\end{proof}
\begin{proof}
We now complete the proof of Proposition \ref{hard.prop.10}.
We want to prove
that $\| H \| \leq 2$.
By Lemma  \ref{lemd1}, 
 $\norm{(1-H)}\le 1$. Therefore,
\[
\norm{H} = \norm {1 -(1-H)} \le \norm{1} + \norm{1-H }\le 1+1 =2.
\]
\end{proof}
\section{An asymptotic M\"untz-Sz\'asz Theorem}
\label{sece}

Let $S$ be a subset of the non-negative integers. When is the linear span of the monomials $\{ x^n : n \in S \}$ dense?
The  M\"untz-Sz\'asz Theorem, proved by M\"untz \cite{mu14} and Sz\'asz \cite{sza16}, answers this question in both $L^2[0,1]$ and $C[0,1]$, the continuous functions on $[0,1]$. The answer is basically the
same in both cases, but the constant function $1$ plays a special role in $C[0,1]$, since it cannot be
approximated by any polynomial that vanishes at $0$ (which all the other monomials do).
\begin{thm}\label{assym.thm.5}
{\bf (M\"untz-Sz\'asz Theorem)}
(i) The linear span of $ \{ x^n : n \in S \}$ is dense in $L^2[0,1]$ if and only if
\be
\label{eqe10}
\sum_{n \in S} \frac{1}{n+1} \= \infty .
\ee
(ii) The linear span of  $\{ x^n : n \in S \}$ is dense in $C[0,1]$ if and only if $0 \in S$ and \eqref{eqe10} holds.
\end{thm}

What happens if the approximants come from a set of linear combinations of monomials that is losing as well as gaining members?
Fix an increasing sequence $\{N_n\}$ of natural numbers and for each $n$ define
\[
S_n=\{n,n+1,\ldots,n+N_n\} \qquad \text{and}
\]
\[
\m_n= \spn \  \{x^{n},x^{n+1},\ldots,x^{n+N_n}\}.
\]
For each $n$ let $\rho_n$ denote the fraction of the non-negative integers less than or equal to $n+N_n$ that do not lie in $S_n$, i.e.,
\[
\rho_n = \frac{n}{n+N_n + 1}.
\]
Finally, with this setup, let
\be\label{assym.10}
\m_\infty = \set{f\in \ltwo}{\lim_{n\to \infty} \dist (f,\m_n) = 0}.
\ee

We wish to characterize $\m_\infty$ (Theorem \ref{assym.thm.10} below). We shall 
 follow M\"untz's original proof of Theorem \ref{assym.thm.5}
\cite{mu14}. 
 His argument involved an ingenious calculation using Theorem \ref{gram.thm.10} and the following venerable formula of Cauchy \cite{cau41}.

\begin{thm}\label{Cauchy}{\bf (The Cauchy Determinant Formula)}
If $M$ is the $N \times N$ Cauchy matrix defined by the formula
\[
M=
\begin{bmatrix}
\frac{1}{x_1-y_1}&\frac{1}{x_1-y_2}&\ldots&\frac{1}{x_1-y_N}\\
\frac{1}{x_2-y_1}&\frac{1}{x_2-y_2}&\ldots&\frac{1}{x_2-y_N}\\
\vdots&\vdots&\ddots&\vdots\\
\frac{1}{x_N-y_1}&\frac{1}{x_N-y_2}&\ldots&\frac{1}{x_N-y_N}
\end{bmatrix},
\]
where for all $i$ and $j$, $x_i \not= y_j$, then
\[
\det M = \frac{\prod\limits_{1\le j<i \le N}(x_i-x_j)(y_j-y_i)}{\prod\limits_{1\le i,j\le N}(x_i-y_j)}.
\]
\end{thm}

We need two more auxiliary results. 

\begin{prop}\label{leme5} {\bf (Baby Brodskii-Donoghue Theorem)}
Let $\M$ be a closed subspace of $L^2[0,1]$ that is invariant under both $M_x$ and $V$.
Then $\M = L^2[s,1]$ for some $s$ between $0$ and $1$.
\end{prop}
\begin{proof}
Note that the constant function 1 has the unique representation in $\ltwo$,
\be\label{assym.40}
1=f+g
\ee
where $f \perp \M$ and $g \in \M$. The fact that $\M$ is $M_x$ invariant, implies that $pg \in \M$ whenever $p$ is a polynomial\footnote{Note that it is also true that $pf \in \M^\perp$ whenever $p$ is a polynomial, since $M_x$ is self-adjoint.},
 it follows that $f \perp pg$ whenever $p$ is a polynomial. But then $f\bar g \perp p$ for all polynomials which implies that
\be\label{assym.50}
f\bar g =0 
\ee

For a Lesbesgue measurable set $E \subseteq [0,1]$ we let $\chi_E$ denote the characteristic function of $E$, i.e., the function  defined by
\[
\chi_E(x)=\twopartdef{0}{x\not\in E}{1}{x \in E}.
\]
We observe that \eqref{assym.40} and \eqref{assym.50} imply that there exists a partition of $[0,1]$ into two measurable sets $F$ and $G$ such that $f=\chi_F $ and $g=\chi_G$. We define a parameter $s\in[0,1]$ by setting
\be\label{assym.55}
s = \sup \set{x}{g(t) =0 \text{ for a.e. } t \in [0,x]}.
\ee
Notice that with this definition, we have that
\be\label{assym.60}
Vg(x) = 0 \text{ for a.e. }\in [0,s]\ \ \ \ \text{ and }\ \ \ \ 
Vg(x) > 0 \text{ for a.e. }\in [s,1].
\ee

Since $g \in \M$, we have  $Vg \in \M$. Also, recall that $f\in \M^\perp$. Therefore, using \eqref{assym.60} we see that $F \subseteq [0,s]$. In light of \eqref{assym.40}, this implies $[s,1]\subseteq G$, which in turn, implies via \eqref{assym.55} that
\be\label{assym.70}
f=\chi_{[0,s]}\ \ \text{ and }\ \ g=\chi_{[s,1]}.
\ee

As $pf \in \M^\perp$ and $pg \in \M$ whenever $p$ is a polynomial, it follows immediately from \eqref{assym.70} and the fact that the polynomials are dense in both ${{\rm L}^2([0,s])}$ and $\ltwos$, that
\[
{{\rm L}^2([0,s])} \subseteq \M^\perp \ \ \text{ and }\ \ \
\ltwos \subseteq \M.
\]
Hence, we have that both ${{\rm L}^2([s,1])} \supseteq \M $ and $\ltwos \subseteq \M$, so that $\ltwos = \M$, as was to be proved.
\end{proof}

We call Proposition \ref{leme5} the Baby Brodskii-Donoghue Theorem
because Brodskii and Donoghue independently proved the far deeper fact that the only closed invariant subspaces of $V$
are $L^2[s,1]$ \cite{br57,don57}. The operator $M_x$ has other invariant subspaces. Indeed the ideas in the  preceding proof 
can be adapted to show that
the invariant subspaces of $M_x$ are the spaces $\{ f \in L^2[0,1] : f(x) = 0 {\rm\ a.e.\ on\ }F\}$, where
$F$ is any measurable subset of $[0,1]$.

\begin{lem}\label{assym.lem.10}
If $\m_\infty$ is as in \eqref{assym.10}, then there exists $s\in[0,1]$ such that $\m_\infty = \ltwos$.
\end{lem}
\begin{proof}
Observe first that if
\[
p(x) = \sum_{k=n}^{n+N_n}a_k x^k \in \m_n,
\]
then
\[
M_xp(x)=xp(x) = \sum_{k=n+1}^{n+N_n+1}a_{k-1} x^{k} \in \m_{n+1}.
\]
Hence,
\[
M_x \m_\infty \subseteq \m_\infty.
\]
 Likewise,
\[
V\m_\infty \subseteq \m_\infty.
\]
Now the result follows from Lemma \ref{leme5}.
\end{proof}
\begin{thm}\label{assym.thm.10}
{\bf (Asymptotic M\"untz-Sz\'asz Theorem)}
Let $S_n=\{n,n+1,\ldots,n+N_n\}$, let $\m_n= \spn \  \{x^{n},x^{n+1},\ldots,x^{n+N_n}\}$, and let $
\rho_n = \frac{n}{n+N_n + 1}.
$
If 
\[
\lim_{n \to \infty} \rho_n = \rho,
\]
then
\[
\m_\infty = {{\rm L}^2([\rho^2,1])}.
\]
\end{thm}
\begin{proof}
By Lemma \ref{assym.lem.10} there exists $s \in [0,1]$ such that
\[
\m_\infty = \ltwos.
\]
Noting that
\[
\dist(1,\ltwos)= \sqrt s,
\]
we see that the theorem will follow if we can show that
\[
\dist(1,\m_\infty) = \rho,
\]
or equivalently that
\be\label{assym.80}
\lim_{n \to \infty} \dist(1,\m_n) = \rho.
\ee
Now fix $N+1$ distinct real numbers $\alpha_0,\alpha_1,\ldots,\alpha_N\in (-\tfrac12,\infty)$. In Theorem \ref{Cauchy} if for $i,j =1,2,\ldots,N$ we let $x_i=\alpha_i+\tfrac12$ and $y_j=-(\alpha_j+\tfrac12)$ we obtain that
\[
\det G(x^{\alpha_1},\ldots,x^{\alpha_N}) = \frac{\prod\limits_{1\le j<i \le N}(\alpha_i-\alpha_j)^2}{\prod\limits_{1\le i,j\le N}(\alpha_i+\alpha_j+1)}.
\]
Likewise,
\[
\det G(x^{\alpha_0},x^{\alpha_1},\ldots,x^{\alpha_N}) = \frac{\prod\limits_{0\le j<i \le N}(\alpha_i-\alpha_j)^2}{\prod\limits_{0\le i,j\le N}(\alpha_i+\alpha_j+1)}.
\]
Therefore,
\[
\frac{\det G(x^{\alpha_0},x^{\alpha_1},\ldots,x^{\alpha_N})}{\det G(x^{\alpha_1},\ldots,x^{\alpha_N})}=\frac{1}{2\alpha_0+1}\ \frac{\prod_{i=1}^N (\alpha_i-\alpha_0)^2}{\prod_{i=1}^N (\alpha_i+\alpha_0+1)^2}
\]
Hence, using Theorem \ref{gram.thm.10} we get
\begin{eqnarray*}
\dist(1,\m_n)^2 &\=& 
\frac{\det G(x^0, x^n, \dots, x^{n+N_n})}{\det G( x^n, \dots, x^{n+N_n})} \\
&=&
\frac{\prod_{i=0}^{N_n} (n+i)^2}{\prod_{i=0}^{N_n} (n+i+1)^2}\\
&=& (\frac{n}{n+N_n+1})^2 \\
&=& \rho_n^2 .
\end{eqnarray*}
Equation \eqref{assym.80} now follows.
\end{proof}

\section{The Bernstein Conundrum: Asymptotic M\"untz-Sz\'asz Theorem for $C[0,1]$}

The $C[0,1]$ M\"untz-Sz\'asz Theorem can be deduced from the $L^2$ version.
What about the asymptotic version?
Let $\m_n$ be as in Section \ref{sece}, and let $\cmi$ be
\be\label{assym.11}
\cmi \= \set{f\in C[0,1]}{\lim_{n\to \infty} \dist (f,\m_n) = 0},
\ee
where in this section all distances are with respect to the supremum norm\footnote{This means $\| f \|_{C[0,1]} = \sup_{0 \leq x \leq 1} |f(x) |$.}.

One way to prove the Weierstra{\ss} approximation theorem is to use the Bernstein polynomials. 
For each $n$, these are the $n+1$ polynomials defined by
\[
b_{k,n} (x) \= {{n} \choose{k}} x^k (1-x)^{n-k} .
\]
Bernstein proved in 1912 \cite{ber12} that for every continuous function $f \in C[0,1]$, the polynomials
\be
\label{eqe16}
p_n (x) \= \sum_{k=0}^n f(\frac{k}{n})\  b_{k,n} (x)
\ee
converge uniformly on $[0,1]$ to $f$.

As the lowest order term in $b_{k,n}$ is $x^k$, if $f$ vanished on $[0,\rho_n]$ and
one used the Bernstein formula \eqref{eqe16} to approximate it, the corresponding polynomial
$p_{n+ N_n + 1}$ would lie in the span of $\{ x^{n+1}, \dots, x^{n+ N_n +1} \}$ which is in $x \m_n \subseteq \m_{n+1}$.
 So one immediately gets that $\cmi$ contains all the continuous
functions that vanish on $[0,\rho]$.

This construction seems natural, and 
 could lead one to suspect that $\cmi$ should be the functions
that vanish on $[0,\rho]$. However, Theorem \ref{assym.thm.11} shows this is incorrect.

\begin{thm}\label{assym.thm.11}
{\bf (Asymptotic M\"untz-Sz\'asz Theorem, Continuous Case)}
If 
\[
\lim_{n \to \infty} \rho_n = \rho,
\]
then
\[
\cmi  \=  \{ f \in C[0,1] | f =0 {\rm\ on \ } [0,\rho^2] \}.
\]
\end{thm}
\begin{proof}
As the supremum norm is larger than the $L^2$ norm, we have
\begin{eqnarray*}
\cmi  &\ \subseteq \  &\m_\infty \cap C[0,1] \\
&=& \{ f \in C[0,1] | f =0 {\rm\ on \ } [0,\rho^2] \}.
\end{eqnarray*}
For the reverse inclusion, notice that it follows from Cauchy-Schwarz that the Volterra
operator is a bounded linear map from $L^2[0,1]$ into $C[0,1]$. (Indeed, if $g \in L^2[0,1]$, we get
that $Vg$ satisfies a H\"older continuity condition of order $\frac{1}{2}$.)
  
Let $f$ be a $C^1$ function that vanishes on $[0,\rho^2]$. Then $f = Vg$, where $g = f'$. 
By Theorem \ref{assym.thm.10}, there are polynomials $p_n \in \m_n$ that
converge in $L^2$ to $g$. Then $V p_n$ converges in $C[0,1]$ to $f$, so $f$ is in $\cmi$. As $\cmi$ is closed, and
the $C^1$ functions that vanish on $[0,\rho^2]$ are dense in the continuous ones, we get 
\[
\{ f \in C[0,1] | f =0 {\rm\ on \ } [0,\rho^2] \} \ \subseteq \ \cmi .
\]
\end{proof}
\begin{ques} Can one prove Theorem \ref{assym.thm.11} directly using Bernstein approximation?
\end{ques}

\section{Proof of Theorem \ref{thm.over.10}}
\label{secf}
\begin{proof}
Assume that $T$ is a monomial operator of order $m$, let $s \in (0,1)$, and fix $f\in \ltwos$. We wish to prove that  $Tf \in \ltwos$.

Choose an increasing sequence of natural numbers $\{N_n\}$ such that
\[
\lim_{n \to \infty} \frac{n}{n+N_n} = \sqrt s.
\]
By Theorem \ref{assym.thm.10}, there exists a sequence of polynomials $\{p_n\}$ where for each $n$, $p_n$ has the form
\[
p_n(x)=\sum_{k=n}^{n+N_n}c_k x^k
\]
and such that
\[
p_n \to f\ \text{ in }\ \ltwo.
\]
As $T$ is bounded,
\[
Tp_n \to Tf\ \text{ in }\ \ltwo.
\]
Also, as $T$ is a monomial operator of order $m$, for each $n$, $Tp_n$ has the form
\[
Tp_n(x)= x^m \sum_{k=n}^{n+N_n}d_k x^k.
\]
For any $s > 0$, multiplication by $x^m$ is a bounded invertible map from $\ltwos$ to itself.
Therefore
\[
\sum_{k=n}^{n+N_n}d_k x^k \] converges, as $n \to \infty$, to some function $g(x)$, which
by   Theorem \ref{assym.thm.10} is in $\ltwos$.
So $Tf = x^m g(x)$, and lies in
in $\ltwos$.
\end{proof}


\bibliography{../references_uniform_partial}
\end{document}